\documentclass[12pt,leqno,draft]{article}
\usepackage{amsfonts}
\usepackage{amsmath, amsthm, amsfonts, amssymb, color}
\usepackage{mathrsfs}
\usepackage{color}
\usepackage{stmaryrd}
\makeatletter
\newcommand{\Spvek}[2][r]{%
  \gdef\@VORNE{1}
  \left(\hskip-\arraycolsep%
    \begin{array}{#1}\vekSp@lten{#2}\end{array}%
  \hskip-\arraycolsep\right)}

\def\vekSp@lten#1{\xvekSp@lten#1;vekL@stLine;}
\def\vekL@stLine{vekL@stLine}
\def\xvekSp@lten#1;{\def\temp{#1}%
  \ifx\temp\vekL@stLine
  \else
    \ifnum\@VORNE=1\gdef\@VORNE{0}
    \else\@arraycr\fi%
    #1%
    \expandafter\xvekSp@lten
  \fi}
\makeatother

\newtheorem{thm}{Theorem}[section]

\newtheorem{lem}[thm]{Lemma}

\newtheorem{rem}[thm]{Remark}
\theoremstyle{definition}
\newtheorem{defn}{Definition}[section]
\newcommand{\scr}[1]{\mathscr #1}
\definecolor{wco}{rgb}{0.5,0.2,0.3}

\numberwithin{equation}{section} \theoremstyle{remark}

\newcommand{\ua}{\uparrow}

\title{{\bf     Harnack   Inequalities for $G$-SDEs with Multiplicative Noise}
\footnote{Supported in
 part by   NNSFC (11801403, 11801406).}
}
\author{
{\bf   Fen-Fen  Yang }\\
\footnotesize{  Center for Applied Mathematics, Tianjin University, Tianjin 300072, China}\\
\footnotesize{  yangfenfen@tju.edu.cn}}
\begin{document}
\allowdisplaybreaks
\def\R{\mathbb R}  \def\ff{\frac} \def\ss{\sqrt} \def\B{\mathbf
B}
\def\N{\mathbb N} \def\kk{\kappa} \def\m{{\bf m}}
\def\ee{\varepsilon}\def\ddd{D^*}
\def\dd{\delta} \def\DD{\Delta} \def\vv{\varepsilon} \def\rr{\rho}
\def\<{\langle} \def\>{\rangle} \def\GG{\Gamma} \def\gg{\gamma}
  \def\nn{\nabla} \def\pp{\partial} \def\E{\mathbb E}
\def\d{\text{\rm{d}}} \def\bb{\beta} \def\aa{\alpha} \def\D{\scr D}
  \def\si{\sigma} \def\ess{\text{\rm{ess}}}
\def\beg{\begin} \def\beq{\begin{equation}}  \def\F{\scr F}
\def\Ric{\text{\rm{Ric}}} \def\Hess{\text{\rm{Hess}}}
\def\e{\text{\rm{e}}} \def\ua{\underline a} \def\OO{\Omega}  \def\oo{\omega}
 \def\tt{\tilde} \def\Ric{\text{\rm{Ric}}}
\def\cut{\text{\rm{cut}}} \def\P{\mathbb P} \def\ifn{I_n(f^{\bigotimes n})}
\def\C{\scr C}   \def\G{\scr G}   \def\aaa{\mathbf{r}}     \def\r{r}
\def\gap{\text{\rm{gap}}} \def\prr{\pi_{{\bf m},\varrho}}  \def\r{\mathbf r}
\def\Z{\mathbb Z} \def\vrr{\varrho} \def\ll{\lambda}
\def\L{\scr L}\def\Tt{\tt} \def\TT{\tt}\def\II{\mathbb I}
\def\i{{\rm in}}\def\Sect{{\rm Sect}}  \def\H{\mathbb H}
\def\M{\scr M}\def\Q{\mathbb Q} \def\texto{\text{o}} \def\LL{\Lambda}
\def\Rank{{\rm Rank}} \def\B{\scr B} \def\i{{\rm i}} \def\HR{\hat{\R}^d}
\def\to{\rightarrow}\def\l{\ell}\def\iint{\int}
\def\EE{\scr E}\def\no{\nonumber}
\def\A{\scr A}\def\V{\mathbb V}\def\osc{{\rm osc}}
\def\BB{\scr B}\def\Ent{{\rm Ent}}\def\3{\triangle}\def\H{\scr H}
\def\U{\scr U}\def\8{\infty}\def\1{\lesssim}\def\HH{\mathrm{H}}
 \def\T{\scr T} \def\BE{\bar{\mathbb{E}}}
\maketitle

\begin{abstract}
 The Harnack and $\log$  Harnack  inequalities  for stochastic differential equation driven by $G$-Brownian motion with  multiplicative noise   are derived by means of coupling by change of measure,  which extend the correspongding results derived in \cite{W} under the linear expectations. Moreover, we generalize the gradient estimate  under nonlinear  expectations appeared in \cite{song}.
\end{abstract} \noindent

 \vskip 2cm

\section{Introduction}
For the extensive applications in strong Feller property,   uniqueness of invariant   probability measures, functional inequalities, and  heat kernal estimates, Wang's Harnack inequality has been developed \cite{W}.
To establish  Harnack inequality, Wang  introduced the coupling by change of measures, see   \cite{bao, W1, W4} and references within for details.
However, up to now, most of these papers only focus on  the case of  linear expectation spaces.
Song \cite{song} firstly derived the gradient estimates for nonlinear diffusion semigroups by using the method of Wang's coupling by change of measure,
after Peng \cite{peng2, peng1} established the systematic  theory of $G$-expectation theory, $G$-Brownian motion and stochastic differential equations driven by $G$-Brownian
motion ($G$-SDEs, in short).
  Subsequently, Yang \cite{Yang} generalized the theory of Wang's Harnack inequality  and its applications  to nonlinear expectation framework, where the
  noise   is additive. Moreover, Wang's Harnack inequality and gradient estimates are also proved for the degenerate (functional) case in \cite{XY}.
  An interesting question is whether it can be generalized
  to the form of multiplicative noise. The answer is positive as some of the results are showed in \cite{song}, whereas neither the form of  $G$-SDEs  with  the term of $\d \langle B^i, B^j\rangle_t$, nor the Harnack    inequality studied, where $B_t$ is a $d$-dimensional
$G$-Brwonian motion, and $\langle
 B^i,B^j\rangle_t$ stands for the mutual variation process of the
 $i$-th component $B^i_t$ and the
 $j$-th component $B^j_t$.  In this paper, we will improve and extend the above assertions to the multiplicative noise.
    Consider the following  $G$-SDE
\begin{equation}\label{0}
\d X_t=b(t, X_t) \d t+\sum_{i,j=1}^dh_{ij}(t,X_t)\d \langle B^i, B^j \rangle_t+\sum_{i=1}^d\sigma_i(t, X_t) \d B^i_t,
\end{equation}
  where $b, h_{ij}=h_{ji}:[0,T]\times   \R^d\to \R^d$ and $\si:[0,T]
\times\R^d\to \R^d\otimes \R^d$.
 We aim to establish the Harnack   inequality for the $G$-SDE \eqref{0}.  In addition, we also prove  the gradient estimate. To this end, we firstly recall some basic facts on the $G$-expectation and  $G$-Brownian motion.

For a positive integer $d$, let $(\R^d, \<\cdot,\cdot\>,|\cdot|)$ be the $d$-dimensional Euclidean space, $\mathbb{S}^d$
the collection of all symmetric $d\times d$-matrices.
For any fixed $T>0$, $$\OO_T=\{\omega|[0,T]\ni
t\mapsto\omega_t\in\R^d \mbox{ is continuous with
}\omega(0)={ 0} \}$$ endowed with the uniform form.  Let
$B_t(\omega)=\omega_t, \omega\in\OO_T, $ be the canonical process.
Set
  $$L_{ip}(\Omega_T)=\{\varphi(B_{t_1}, \cdots, B_{t_n}),
  n\in\mathbb{N},
   t_1,\cdot \cdot \cdot, t_n\in [0,T],\varphi\in C_{b,lip}(\R^{d}\otimes\R^n)\},$$
where  $C_{b,lip}(\mathbb{R}^{d}\otimes\R^n)$ denotes  the set of
bounded Lipschitz functions.  Let $G: \mathbb{S}^d \to \R $ be a monotonic, sublinear   and homogeneous function; see e.g.
\cite[p16]{peng4}.   Now we give the construction of  $G$-expectation which is also used  in   \cite{RY}.
For any $\xi \in L_{ip}(\Omega_T)$,  i.e.,
 $$\xi(\omega)=\varphi(\omega(t_1),\cdot \cdot \cdot,\omega(t_{n})), \ 0=t_0< t_1<\cdot \cdot \cdot<t_n=T,$$
  the conditional $G$-expectation is  defined by
  $$
 \bar{\mathbb{E}}_t[\xi]:=u_k(t,\omega(t);\omega(t_1),\cdot \cdot \cdot,\omega(t_{k-1}) ),  \  \xi \in L_{ip}(\Omega_T), \ t \in [t_{k-1}, t_k), \ k=1, \cdot\cdot\cdot, n,
 $$
 where $(t,x)\mapsto u_k(t,x; x_1, \cdot\cdot\cdot,x_{k-1})$, $k=1,\cdot \cdot\cdot,n$, solves the following $G$-heat equation
\begin{equation}\label{Gheat}
 \begin{cases}
 \partial_tu_k+G(\partial_x^2u_k)=0, \ (t,x) \in [t_{k-1}, t_k)\times \R^d, \ k=1,\cdot \cdot\cdot,n,\\
 u_k(t_k,x;x_1,\cdot\cdot\cdot,x_{k-1})= u_{k+1}(t_k,x;x_1,\cdot\cdot\cdot,x_{k-1},x_k), \ k=1, \cdot\cdot\cdot, n-1,\\
  u_n(t_n,x;x_1,\cdot\cdot\cdot,x_{n-1})=\varphi(x_1,\cdot\cdot\cdot,x_{n-1},x), \ k=n.
 \end{cases}
 \end{equation}
The corresponding $G$-expectation of $\xi$ is defined by $\bar{\mathbb{E}}[\xi]=\bar{\mathbb{E}}_0[\xi]$.

According to \cite{peng4}, there exists a bounded, convex, and closed subset $\Gamma\subset \mathbb{S}^d_+$ such that
\begin{equation}\label{GA}
 G(A)=\frac{1}{2}\sup _{Q\in \Gamma}\mbox{trace}[AQ], \ A\in\mathbb{S}^d.
\end{equation}
In particular,
fix $\underline{\sigma}, \overline{\sigma}\in \mathbb{S}_+^d$ with $\underline{\sigma}<\overline{\sigma}$, let
$\Gamma=[\underline{\sigma}^2, \bar{\sigma}^2],$ then
\begin{equation}\label{G(A)}
 G(A)=\frac{1}{2}\sup _{\gamma\in [\underline{\sigma}, \bar{\sigma}]}\mbox{trace}(\gamma^2 A), \ A\in\mathbb{S}^d.
\end{equation}
Denote  $L_G^{p}(\Omega_T)$  be the completion of $L_{ip}(\Omega_T)$  under the norm $(\bar{\mathbb{E}}[|\cdot|^p])^{\frac{1}{p}}$, $p\geq1. $
\begin{thm} (\cite{Denis, peng4}) There exists a weakly compact subset $\mathcal{P} \subset M_1(\Omega_T)$, the set of probability
measures on $(\Omega_T , \mathcal{B}(\Omega_T))$, such that
\begin{equation}\label{rep}
\BE[\xi] = \sup_{P \in \mathcal{P}}\E_P[\xi] \ for  \ all  \ \xi \in L_{G}^1(\Omega_T).
\end{equation}
$\mathcal{P}$ is called a set that represents $\BE$.
\end{thm}
Let $\mathcal{P}$ be a weakly compact set that represents $\BE$. For this $\mathcal{P}$, we define capacity
\begin{equation}\label{c}
c(A) = \sup_{P\in\mathcal{P}}P(A), A \in \mathcal{B}(\Omega_T).
\end{equation}
$c$ defined here is independent of the choice of $\mathcal{P}$.
\begin{rem}
\begin{itemize}\label{line}
  \item [(i)] Let $(\Omega^0, \mathcal{F}^0, P^0)$ be a probability space and $\{W_t\}$ be a $d$-dimensional
Brownian motion under $P^0$. Let  $F^0 =\{\mathcal{F}^0_t\}_{t\geq 0}$ be the augmented filtration generated by W. \cite{Denis} proved that
$$\mathcal{P}_M := \{P_h\ | \  P_h= P^0 \circ X^{-1}, X_t = \int_0^t h_s\d W_s, h_s\in L^2_{ F^0}([0,T];\Gamma^{\ff{1}{2}}) \}$$
is a set that represents $\BE$, where $\Gamma^{\ff{1}{2}} := \{\gamma^{\ff{1}{2}} \ | \ \gamma \in \Gamma \}, $ is the set in the representation
of $G(\cdot)$ in the formula \eqref{GA} and $L^2_{ F^0}([0,T];\Gamma^{\ff{1}{2}})$ is the set of $F^0$-progressive measurable
processes with values in $\Gamma^{\ff{1}{2}}$.
  \item [(ii)]  For the $1$-dimensional case, $L^2_{ F^0}([0,T];\Gamma^{\ff{1}{2}})$ reduces to the form below:
$$\{ h \ |  \ h \ is \ an \ progressive \ measurable  \ process \ w.r.t.\  F^0 \ and  \ \underline{\sigma} \leq  |h_s|  \leq \bar{\sigma}\}.$$
\end{itemize}

\end{rem}
\begin{defn}
We say a set $A\subset \Omega_T$ is $c$-polar if $c(A) = 0$. A property holds quasi-surely ($c$-q.s. for short) if it holds
outside a $c$-polar set.
\end{defn}

\begin{defn}
\begin{itemize}\label{qc}
  \item[(1)]  We say that a map $\xi(\cdot): \Omega_T   \rightarrow  \R $ is quasi-continuous if for all  $\epsilon > 0$, there
exists an open set $G$ with $c(G) < \epsilon $ such that $\xi(\cdot)$ is continuous on $G^c$.
  \item[(2)]   We say that a process $M_{\cdot}(\cdot) :  \Omega_T \times [0,T]  \rightarrow  \R$ is quasi-continuous if for all  $\epsilon > 0$, there
exists an open set $G$ with $c(G) < \epsilon $ such that $M_{\cdot}(\cdot)$ is continuous on $G^c\times [0, T].$
 \item[(3)] We say that a random variable $X : \Omega_T   \rightarrow  \R $ has a quasi-continuous version if there exists a quasi-continuous
function $Y : \Omega_T   \rightarrow  \R $ such that $X = Y$, $c$-q.s.

\end{itemize}
\end{defn}
\begin{rem}
Note that a quasi-continuous process defined here  is different from \cite{HWZ}.
\end{rem}
According to \cite{Denis},
$$
L_G^{p}(\Omega_T) =\{ X\in L^0(\Omega_T)\ |\lim_{N\to \infty}\bar{\E}[|X|^p1_{|X|\geq N}]=0  \ and \ X \ has \ a \  quasi-continuous \ version\},
$$
where $ L^0(\Omega_T)$ denotes  the space of all $ \mathcal{B}(\Omega_T)$-measurable real function.

In the paper, we discuss the property of distribution for the solution  $X_t$  in \eqref{0}, a polar set does not affect the result, so in the following parts, we did not distinguish the quasi-continuous version and  itself  any more.

\begin{thm} \label{Mono}(Monotone Convergence Theorem)
\cite[Theorem 10, Theorem 31]{Denis} Let $\mathcal{P}$ be weakly compact that represents $\BE$.
\begin{itemize}
  \item[(1)] Suppose $\{X_n\}_{n\geq1}, X \in L^0(\Omega_T),$ $X_n \uparrow X,$ $c$-q.s. and $\E_P [X_1^-] < \infty$ for all $P \in  \mathcal{P}$. Then $\BE[X_n] \uparrow\BE[X]$.

  \item[(1)]  Let $\{X_n\}_{n=1}^\infty \subset L_G^1(\Omega_T)$ be such that $X_n \downarrow X,$
$c$-q.s.. Then $\BE[X_n] \downarrow\BE[X]$.

\end{itemize}
\end{thm}
\begin{rem}
We stress that in this theorem $X$ does not necessarily belong to $L_G^1(\Omega_T)$.
\end{rem}
Let
\begin{equation*}\label{equa11} M_G^{p,0}([0,T])=
\Bigg\{\eta_t~ | ~\eta_t=\sum_{j=0}^{N-1} \xi_{j} I_{[t_j, t_{j+1})},\xi_{j}\in L_{G}^p(\Omega_{t_{j}}),
 N\in\mathbb{N},\ 0=t_0<t_1<\cdots <t_N=T \Bigg\}.
\end{equation*}
For  $p\geq1$, let   $M_G^{p}([0,T])$  and $H_G^{p}([0,T])$ be the completion of $M_G^{p,0}([0,T])$ under the following norm
 $$\|\eta\|_{M_G^{p}([0,T])}=\left[\bar{\mathbb{E}}\left(\int_{0}^{T}|\eta_{t}|^{p}dt\right)\right]^{\frac{1}{p}},  \
 \|\eta\|_{H_G^{p}([0,T])}=\left[\bar{\mathbb{E}}\left(\int_{0}^{T}|\eta_{t}|^{2}dt\right)^{\ff{p}{2}}\right]
 ^{\frac{1}{p}},$$ respectively.
 Denote by $[M_G^p([0,T])]^d$, $[H_G^p([0,T])]^d$   all $d$-dimensional stochastic processes $\eta_t=(\eta_t^1, \cdot \cdot \cdot, \eta_t^d),$ $\xi_t=(\xi_t^1, \cdot \cdot \cdot, \xi_t^d)$, $t\geq0$
 with $\eta_t^i\in M_G^p([0,T]), \xi_t^i\in H_G^p([0,T])$, respectively.
\begin{defn}
A process $X = \{X_t | t\in[0,T]\}$ is called a $G$-martingale if for each
$t\in[0,T]$, we have $X_t \in  L^1_G(\Omega_t) $ and
$$ \BE_s[X_t] = X_s \ in \  t\in[0,T].$$
We call $X$ a symmetric $G$-martingale if both $X$ and $-X$ are $G$-martingales.
\end{defn}
\begin{rem} For  $\eta\in M_G^1([0,T])$,  it's easy to see that the process $\int_0^t \eta_s(\omega) \d s$  has a $c$-quasi continuous version. Also, \cite{song2} shows that any G-martingale has a c-quasi continuous version.
\end{rem}
Let $B_t$  be a $d$-dimensional $G$-Brownian motion, then $G(A)=\frac{1}{2}\bar{\mathbb{E}}[\langle AB_1,B_1\rangle]$,  $A\in \mathbb{S}^d$.
In particular, for 1-dimensional $G$-Brownian motion $(B_t)_{t\ge0}$, one has
$G(a)=(\overline{\sigma}^2a^+-\underline{\sigma}^2a^-)/2, a\in \mathbb{R},  $ where $\overline{\sigma}^2: = \bar{\mathbb{E}} [B^2_1]\geq -\bar{\mathbb{E}} [-B^2_1 ]= :\underline{\sigma}^2>0$.

Let $\langle B\rangle_t=(\langle B^{i},B^{j}\rangle_t)_{1\leq i,j\leq d}, 0\leq t\leq T$, which is defined by
\begin{equation}\label{B}
   \langle B\rangle_t=B_t^iB_t^j-\int_0^tB_s^i\d B_s^j-\int_0^tB_s^j\d B_s^i.
\end{equation}
To establish the Wang's  Harnack inequality,   $G$-Girsanov's transform plays a crucial role,
 the following  results is taken from \cite{Osuka,Xu}.
 For $ \eta \in [M^2_G ([0, T])]^d$, let
\begin{align}\label{Gir}
&M_t=\exp \left\{  \int_0^t \langle \eta_s,   {\d}B _s\rangle -\frac{1}{2}\int_0^t \langle \eta_s,  ({ \d}\langle B \rangle_s\eta_s)\rangle\right \}, \nonumber\\
 &\hat{B} _t=B_t-\int_0^{t}({\d}\langle B \rangle _s \eta_s),  \ t\in[0,T],
\end{align}
where $({\d}\langle B \rangle _s \eta_s)=\left(\sum_{j=1}^d\eta_s^j \d \langle B^i,B^j\rangle_s\right)_{1\leq i\leq d}.$

\begin{lem} \label {lema2.10} {\rm (\cite{Osuka,Xu})}
 If $ \eta \in [M^2_G ([0, T])]^d$ satisfies $G$-Novikov's condition, i.e., for some $\epsilon_0 > 0$, it holds that
 \begin{equation}\label{Novikov}
 \bar{\E}\left[ \exp \left\{ \left( \frac{1}{2}+\epsilon_0 \right )\int_0^T\langle \eta_s, ({\d}\langle B \rangle_s\eta_s)\rangle\right \}  \right]<\infty,
\end{equation}
then the process $M$ is a symmetric $G$-martingale.
\end{lem}
\begin{lem}  \label{lemma3}{\rm(\cite{Osuka}) ($G$-Girsanov's formula)}
Assume that there exists $\sigma_0 > 0$ such that
 $$
 \gamma  \geq \sigma_0I_d \  \ \  for \  \  \ all  \ \ \gamma \in \Gamma,
 $$
 and that $M$ is a symmetric $G$-martingale on $(\Omega_T, L^1_G (\Omega_T), \bar{\E}).$ Define a sublinear expectation $\hat{\E}$ by
 $$\hat{\E}[X]=\bar{\E}[XM_T], \ \ \ X \in  \hat{L}_{ip}(\Omega_T),$$
 where
  $\hat{L}_{ip}(\Omega_T) :=\{ \varphi(\hat{B}_{t_1}, \cdot \cdot \cdot, \hat{B}_{t_n}):n\in \mathbb{N}, t_1,\cdot \cdot \cdot, t_n\in [0,T],\varphi \in C_{b,lip}(\R^{d}\otimes\R^n)\}$.
 Then $\hat{B}_t$  is a $G$-Brownian motion on the sublinear expectation space $(\Omega_T,  \hat{L}_G^1(\Omega_T), \hat{\E})$, where   $\hat{L}_G^1(\Omega_T)$ is the completion of
  $\hat{L}_{ip}(\Omega_T)$ under the norm $\hat{\E}[|\cdot|]$.
  \end{lem}

\begin{rem} The Girsanov theorem also  appeared in \cite[Theorem 5.2]{HJPS}.
\end{rem}
\begin{lem} \label{lem2.1}
For $\hat{B}$ in \eqref{Gir},  then $c$-q.s., $\<\hat{B}\>_t= \< {B}\>_t,$  $t\in[0,T].$
\end{lem}

\begin{proof}
For any $P\in \mathcal{P}$, it holds that
$$P\{\<\hat{B}\>_t\neq \< {B}\>_t, \ t\in[0,T]\}=0.$$
 By \eqref{c}, we have
 $$c\{\<\hat{B}\>_t\neq \< {B}\>_t, \ t\in[0,T]\}=\sup_{P\in \mathcal{P}}P\{\<\hat{B}\>_t\neq \< {B}\>_t, \ t\in[0,T]\}=0,$$
 which  implies  $c$-q.s., $\<\hat{B}\>_t= \< {B}\>_t,$
 $t\in[0,T].$
\end{proof}
 We aim to establish the following  Harnack-type inequality introduced by Feng-Yu Wang:
\begin{equation}\label{w}
\Phi(\bar{P}f(x))) \leq\bar{P}\Phi(f(y)) e^{\Psi(x,y)}, \  x,y \in \mathbb{R}^d, f\in \mathcal{B}_b^+(\mathbb{R}^d),
\end{equation}
where $\Phi$ is a  nonnegative convex function on $[0,\infty)$ and $\Psi$ is a  nonnegative  function on $\mathbb{R}^d\times \mathbb{R}^d$.
In the   setting of  $G$-SDEs, we establish this type inequality for the associated nonlinear Markov operator $\bar{P}_T$. For simplicity, we consider the case of $d=1$, 
but our results and methods still hold for the case $d > 1$.
 To get our desired results, we give following assumptions on $b, \sigma,$ and $h$ in \eqref{0}.
\begin{itemize}
  \item [(H1)] There exists a constant $K>0$, such that
  $$|b(t,x)-b(t,y)|+|h(t,x)-h(t,y)|+|\sigma(t,x)-\sigma(t,y)|\leq K|x-y|,   \ x, y \in \mathbb{R}, t>0.$$
\item [(H2)] There exist  $\kappa_1, \kappa_2$ with $\kappa_2\geq \kappa_1>0$,  
such that $\kappa_1\leq\sigma(t,x)\leq \kappa_2,$ $x\in \R$, $t>0.$
\end{itemize}
From \cite[Theorem 1.2]{peng4}, under the assumption of  (H1),  for any $x\in\R$,  \eqref{0} has a   unique  solution  in $ M_G^2([0,T] )$.
In what follows,  for $T>0$, we define
\begin{equation*}
\bar{P}_Tf(x)=\bar{\mathbb{E}}f(X_T^x), \  f\in C_b^+(\mathbb{R}),
\end{equation*}
where  $ X_T^x$ solves \eqref{0} with initial value $x$.
\begin{rem}
In order to ensure the term $f(X_T^x)\in M_G^2([0,T] )$, we always assume $  f\in C_b^+(\mathbb{R})$.
\end{rem}
The remainder of the paper is organized as follows. In Section 2, we characterize the quasi-continuity of hitting time for processes of certain forms. Finally,  in Section 3 we present  the Harnack and $\log$  Harnack  inequalities for $G$-SDE \eqref{0}, so that main results  in \cite[Theorem 3.4.1, Chap.3]{W1} are extended to the present $G$-setting. Moreover, the gradient estimate is showed in this section.

\section{Main Results}
 Now we turn to the main result of this section.
 \subsection{Harnack and  log-Harnack  inequalities}
\begin{thm} \label{th1} Assume (H1)-(H2).
\begin{itemize}
  \item [(1)] For any nonnegative $ f \in C_b^+(\mathbb{R})$ and $ T>0,  x, y \in  \mathbb{R},$  it holds that
\begin{equation} \label {th1eq1}
\bar{P}_T \log f(y)\leq\log\bar{P}_Tf(x)+\frac{K\left(2+K+\frac{2}{\underline{\si}^2}\right)|x-y|^2}{2 \ff{\kappa_1^6}{\kappa_2^4}(1-\e^{-\underline{\sigma}^2K\left(2+K+\frac{2}{\underline{\si}^2}\right)T})}.
\end{equation}
 \item [(2)]  For  $p>(1+\ff{\kappa_2^3-\kappa_1\kappa_2^2}{\kappa_1^3})^2$, then
 \begin{equation} \label {th1eq2}
(\bar{P}_T f(y))^p\leq\bar{P}_Tf^p(x)\exp\left\{\frac{\sqrt{p}(\sqrt{p}-1) K\left(2+K+\frac{2}{\underline{\si}^2}\right)|x-y|^2}{4(\kappa_2-\kappa_1)[ {\kappa_1}(\sqrt{p}-1)-C](1-\e^{-\underline{\sigma}^2K\left(2+K+\frac{2}{\underline{\si}^2}\right)T})}\right\},
\end{equation}
holds for  any $x, y \in  \mathbb{R} $ and  $ f \in C_b^+(\mathbb{R})$.
\end{itemize}
\end{thm}
To make the proof easy to follow, let us   divide  the proof into the following aspects.
\subsubsection{Martingale convergence }

To apply $G$-Girsanov's formula in Lemma \ref{lemma3}, we need to check that $M$ is a symmetric $G$-martingale. From Lemma \ref{lema2.10},
 we know that $G$-Novikov's condition is a sufficient condition for $M$ to be a symmetric $G$-martingale. However, if we take this for calculation, the assumptions we impose on $\kappa_1, \kappa_2$ are too strong, thus,
 we propose the notion of  uniform integrability
under a nonlinear expectation \cite{cohen}. We would like  to point out \cite{cohen} discusses the martingale convergence in discrete time, for simplicity, we still use  $\BE$ in this paper instead of the notion in \cite{cohen}.

We define the space $L^1$ as the completion under  $(\bar{\mathbb{E}}[|\cdot|])$ of the set
$$\{X\in \mathcal{H} \ |\ (\bar{\mathbb{E}}[|\cdot|])<\infty\},$$
where $\mathcal{H}$ be a vector lattice of real valued functions defined
on $\Omega$, namely $c\in \mathcal{H}$ for each constant $c$ and $|X|\in \mathcal{H}$  if $X \in \mathcal{H}$.
\begin{defn} Let $K\subset L^1$. $K$ is said to be uniformly integrable (u.i.)
if $ \BE(1_{\{|X|\geq c\}}|X|)$ converges to 0 uniformly in $X \in K$ as $c \rightarrow \infty$.
\end{defn}
\begin{lem}(\cite[Corollary 3.1.1]{cohen}) \label{um}
Let $K \subset L^1$. Suppose there is a positive function
$f$ defined on $[0, \infty[$ such that $\lim_{t\rightarrow\infty} t^{-1}f(t) = \infty$ and $\sup_{X\in K} \BE(f \circ |X|) < \infty$.
Then $K$ is uniformly integrable.
\end{lem}
Let
$$\mathcal{H}^{ext}  = \{ X \in mF| \min \{\E_{P_h}[X^+], \E_{P_h}[X^-] < \infty \} \ for \  all  \ h\in  L^2_{ F^0}([0,T];\Gamma^{\ff{1}{2}}) \},$$
where  $mF$ is the space of  $F_T$-measurable $\R \cup {\pm\infty}$-valued    functions.
According to \cite{cohen},
$$
L_b^{p} =\{ X\in L^0(\Omega_T)\ |\lim_{N\to \infty}\bar{\E}[|X|^p1_{|X|\geq N}]=0 \}.
$$
This does not need to restrict our attention to those random variables admitting a quasi-continuous version compared with the structure of $L_G^{p}(\Omega_T)$. It's clear that $L_G^{p}(\Omega_T)\subset L_b^p.$
\begin{lem}(\cite[Theorem 3.2]{cohen})\label{lem3.2}
Suppose $(X_n)_{n\geq1} \subset L_b^1$, and $X \in \mathcal{H}^{ext}$. Then  $X_n$
converge in $L^1$ norm to $X$ if and only if the collection $(X_n)_{n\geq 1}$ is uniformly
integrable and the $X_n$ converge in capacity to $X$.
Furthermore, in this case, the collection $(X_n)_{n\geq 1} \cup {X} $ is also uniformly
integrable and $X \in L_b^1.$
\end{lem}
\begin{lem}(\cite[Theorem 4.4]{cohen})
 Let $(X_n)_{n\geq1} $ be a  $G$-submartingale with $\sup_k\BE(|X_k|)<\infty$. Then
$X_n\rightarrow X_\infty \in H^{ext}$, q.s..
\end{lem}
\begin{lem}(\cite[Theorem 4.5]{cohen})
 Let $(X_n)_{n\geq1} $ be a uniformly integrable $G$-submartingale. Then
taking $X_\infty = lim_{n\rightarrow\infty}X_n$, the process $(X_n)_{n\geq1 \cup\infty} $ is also a uniformly integrable
$G$-submartingale. In particular, this implies that $X_\infty \in L^1_b$.
\end{lem}
In the following, we aim to extend the convergence theorem for $G$-martingale from discrete time to continuous time.
\begin{thm}\label{cts}
Let $(X_s)_{s\in[0,T)}\subset L_G^1(\Omega_T)$  be a uniformly integrable $G$-martingale. Then
taking $X_T = lim_{t\rightarrow T}X_t$, the process $(X_s)_{s\in[0,T]} $ is also a uniformly integrable
$G$-martingale. In particular, this implies that $X_T \in  L_G^1(\Omega_T)$.
\end{thm}
\begin{proof}
Since $\{X_{T-\ff{T}{n}}\}_{n=1}^\infty$ is a sequence of discrete martingale, we have
$$
\BE_{T-\ff{T}{n}}X_T=X_{T-\ff{T}{n}}.
$$
For any $s\in[0,T)$, there exists a $n\geq 1$, such that
$T-\ff{T}{n}>s$. Moreover,
\begin{align*}
\BE_s X_T&=\BE_s \BE_{T-\ff{T}{n}}X_T \\
 &=\BE_s  X_{T-\ff{T}{n}}\\
 &=X_s,
\end{align*}
where the last step by using the fact that  $(X_s)_{s\in[0,T)}$ is $G$-martingale. This implies that  $(X_s)_{s\in[0,T]}$ is $G$-martingale.
Moreover, the collection $(X_{T-\ff{T}{n}})_{n\geq1}$ is uniformly
integrable and the $X_{T-\ff{T}{n}}$ converge in capacity to $X_T$, then  the $X_t$
converge to $X_T$ in $(\bar{\mathbb{E}}[|\cdot|])$  norm by Lemma \ref{lem3.2}, which proves that $X_T \in  L_G^1(\Omega_T)$.
\end{proof}

To prove Theorem \ref{th1}, we first introduce the construction of coupling by change of measure with multiplicative noise under $G$-setting.
\subsubsection{Construction of the coupling }
In the sequel, we denote 
 $\hat{\sigma}=\sigma^{\ast}(\sigma\sigma^{\ast})^{-1}.$
We use the coupling by change of measures as explained in \cite{W1}.
 For $\alpha\in(0, \ff{2\kappa_1^2}{\kappa_2^2})$, let
  \begin{equation}\label{lambda}
   \lambda_t^\alpha=\frac{\ff{2\kappa_1^2}{\kappa_2^2}-\alpha}{K\left(2+K+\frac{2}{\underline{\si}^2}\right)}\left(1-\e^{\underline{\sigma}^2K\left(2+K+\frac{2}{\underline{\si}^2}\right)(t-T)}\right), \ t\in[0,T].
  \end{equation}
  Then $ \lambda_t^\alpha$  is  smooth and strictly positive on $[0,T)$ such that
\begin{equation}\label{alpha}
\ff{2\kappa_1^2}{\kappa_2^2}-K\left(2+K+\frac{2}{\underline{\si}^2}\right){\lambda_t^\alpha}+\frac{1}{\underline{\sigma}^2}(\lambda_t^\alpha)'=\alpha, \ t\in[0,T].
\end{equation}
 For convenience, we reformulate \eqref{0} as  
\begin{equation}\label{00}
\d X_t=b(t, X_t) \d t+h(t,X_t)\d \langle B \rangle_t+\sigma(t, X_t) \d B_t, \ X_0=x.
\end{equation}
Consider the equation
\begin{equation}
\begin{cases} \label{1}
 \d Y_t=b(t,Y_t)\d t+h(t,Y_t) \d\langle B \rangle _t+\sigma(t,Y_t)\d B_t+\sigma(t,Y_t)g_t\d \langle B \rangle _t,\\
  \ Y_0=y,  \  t\in(0,T),
\end{cases}\end{equation}
where $g_t :=\frac{1}{\lambda_t^\alpha}\hat{\sigma}(t,X_t) (X_t-Y_t) $.

\subsubsection{Extension of $Y$ to $T$}
 Let $s\in[0,T)$ be fixed. By \eqref{0} and \eqref{1},  $ X_t-Y_t$ satisfies the equation below
\begin{align}\label{ito}
 \d (X_t -Y_t)&=(b(t,X_t)-b(t,Y_t))\d t+(h(t,X_t)-h(t,Y_t)) \d\langle B \rangle _t\\
 &\quad+(\sigma(t,X_t)-\sigma(t,Y_t))\d B_t-\sigma(t,Y_t)g_t\d \langle B \rangle _t\nonumber.
\end{align}
Applying It\^{o}'s formula to  $|X_t-Y_t|^2$, we obtain
\begin{align}\label{qc1}
 \d|X_t-Y_t|^2
 &=2\langle X_t-Y_t, b(t,X_t)-b(t,Y_t)\>\d t+2\langle X_t-Y_t, \sigma(t,X_t)-\sigma(t,Y_t)\rangle \d   {B}_t\\\no
 &\quad+2\langle X_t-Y_t, h(t,X_t)-h(t,Y_t)\rangle \d\langle  {B}\rangle_t
+|\sigma(t,X_t)-\sigma(t,Y_t)|^2\d\langle  {B}\rangle_t\\\no
&\quad-2\langle X_t-Y_t, \sigma(t,Y_t)g_t\>\d \<B\>_t
\\\no
 &\leq\left(2K+K^2-\frac{2\kappa_1^2}{\lambda_t^\alpha\kappa_2^2}\right)|X_t-Y_t|^2 \d\langle {B}\rangle_t +2K|X_t-Y_t|^2 \d t\\\no
 &\quad+2\langle X_t-Y_t, \sigma(t,X_t)-\sigma(t,Y_t)\rangle \d   {B}_t\\\no
 &\leq\left(2K+\frac{2K}{\underline{\sigma}^2}+K^2-\frac{2\kappa_1^2}{\lambda_t^\alpha\kappa_2^2}\right)|X_t-Y_t|^2 \d\langle  {B}\rangle_t \\\no
 &\quad+2\langle X_t-Y_t, \sigma(t,X_t)-\sigma(t,Y_t)\rangle \d   {B}_t.\no
\end{align}
 Combining with the expression  \eqref{alpha}, we have
\begin{align*}
 \d\frac{|X_t-Y_t|^2}{\lambda_t^\alpha}
&\leq -\frac{|X_t-Y_t|^2}{(\lambda_t^\alpha)^2}\left(\frac{2\kappa_1^2}{ \kappa_2^2}-2K\lambda_t^\alpha-\frac{2K}{\underline{\sigma}^2}\lambda_t^\alpha-K^2\lambda_t^\alpha+\frac{1}{\underline{\sigma}^2}(\lambda_t^\alpha)'\right)
   \d\langle  {B}\rangle_t,\\
  &\quad+\frac{2}{\lambda_t^\alpha}\langle X_t-Y_t, \sigma(t,X_t)-\sigma(t,Y_t)\rangle \d  {B}_t\\
   &= -\frac{\alpha}{(\lambda_t^\alpha)^2}|X_t-Y_t|^2
   \d\langle  {B}\rangle_t+\frac{2}{\lambda_t^\alpha}\langle X_t-Y_t, \sigma(t,X_t)-\sigma(t,Y_t)\rangle \d  {B}_t.
\end{align*}
Thus,
\begin{align}\label{4'}
 \int_0^s\frac{|X_t-Y_t|^2}{(\lambda_t^\alpha)^2} \d\langle  {B}\rangle_t
 &\leq \int_0^s\frac{2}{\alpha\lambda_t^\alpha}\langle X_t-Y_t, \sigma(t,X_t)-\sigma(t,Y_t)\rangle \d  {B}_t \\
 &\quad-\frac{|X_s-Y_s|^2}{\alpha\lambda_s^\alpha}+\frac{|x-y|^2}{\alpha\lambda_0^\alpha}, \ s \in [0,T).\nonumber
\end{align}
Taking expectation $\bar{\mathbb{E}}$ on both sides of \eqref{4'},  we obtain
\begin{equation}\label{2'}
 \bar{\mathbb{E}}\int_0^s\frac{|X_t-Y_t|^2}{(\lambda_t^\alpha)^2} \d\langle  {B}\rangle_t
 \leq\frac{|x-y|^2}{\alpha\lambda_0^\alpha}, \ s \in [0,T).
\end{equation}
Since $X_t, Y_t \in M_G^2([0,T])$, for any $s\in(0,T),$   $g_t1_{[0,s]}(t) \in M_G^2([0,T])$.
Note that, for any $s\in(0,T),$
\begin{equation*}
 \bar{\mathbb{E}}\int_r^s\frac{|X_t-Y_t|^2}{(\lambda_t^\alpha)^2} \d t
 \leq C_1(s-r),
\end{equation*}
where $C_1$ is a constant.

By  the Monotone Convergence Theorem in  [1] of Theorem \ref{Mono},
 \begin{equation*}
 \bar{\mathbb{E}}\int_r^T\frac{|X_t-Y_t|^2}{(\lambda_t^\alpha)^2} \d t=\lim_{s\rightarrow T}\bar{\mathbb{E}}\int_r^s\frac{|X_t-Y_t|^2}{(\lambda_t^\alpha)^2} \d t
 \leq C_1(T-r).
\end{equation*}
There exists a $\bar{g}\in  M_G^2([0,T]) $ such that $\bar{g}_s=g_s, s\in [0,T).$  In fact,
let    $g_t^n=g_t1_{[0,T-\ff{1}{n}]}(t) $ $\in M_G^2([0,T]),$  then 
it holds that
\begin{align*}
 \bar{\mathbb{E}}\int 1_{[0,T]}|\bar{g}_t-g_t^n|^2 \d t&=  \bar{\mathbb{E}}\int1_{(T-\ff{1}{n}, T]}|\bar{g}_t|^2 \d t\\
 &=  \bar{\mathbb{E}}\int1_{(T-\ff{1}{n}, T)}| {g}_t|^2 \d t\\
&\leq \ff{1}{\kappa_1^2} \bar{\mathbb{E}}\int_{T-\ff{1}{n}}^T\frac{|X_t-Y_t|^2}{(\lambda_t^\alpha)^2} \d t\\
&\rightarrow 0, \ n\rightarrow \infty,
\end{align*}
where the last step uses the fact of [2] in Theorem \ref{Mono}.

 Let  $\bar{Y}_t$ solve the following equation
\begin{equation}\label{1'}
\begin{cases}
 \d Y_t=b(t,Y_t)\d t+h(t,Y_t) \d\langle B \rangle _t+\sigma(t,Y_t)\d B_t+\sigma(t,Y_t)\bar{g}_t\d \langle B \rangle _t,\\
  \ Y_0=y,  \ t\in(0,T],
\end{cases}\end{equation}
 Thus, $Y$ can be  extended to $[0,T]$ as $\bar{Y}$. In the sequel, we still use $Y$ and $g$ instead $\bar{Y}$ and $\bar{g}$.

\subsubsection{Several lemmas}
We first prove the following  Young inequality under $G$-expectation framework.

\begin{lem} \label{Young} {\rm(Young Inequality)} For $g_1, g_2\in L_G^1(\Omega_T)$  with $g_1$, $g_2>0$ and $ {\E}_P[g_1]=1$,   $ \forall \ P \in \mathcal{P},$ 
then
\begin{equation*}
  \bar{\E}[g_1 g_2]\leq \bar{\E}[g_1 \log g_1]+\log \bar{\E}[\e ^{g_2}],
\end{equation*}
where $\mathcal{P}$ is a weakly compact set that represents $\BE$.
\end{lem}
\beg{proof}
For any $ P \in \mathcal{P},$   ${\E}_P$ is a linear expectation,  it holds that
$$ {\E}_P [g_1 g_2]\leq{\E}_P[g_1 \log g_1]+\log {\E}_P[\e ^{g_2}].$$
Since  $\bar{\E}[X]=\sup_{P\in\mathcal{P}} {\E}_P[X]$, $X\in L_G^1(\Omega_T),$ then
\begin{align*}
\bar{\E} [g_1 g_2]&\leq
\sup_{P\in\mathcal{P}}\left\{{\E}_P[g_1 \log g_1]+\log {\E}_P[\e ^{g_2}]\right\}\\
&\leq
 \bar{\E}[g_1 \log g_1]+\sup_{P\in\mathcal{P}}\left\{\log {\E}_P[\e ^{g_2}]\right\}\\
 &\leq
 \bar{\E}[g_1 \log g_1]+  \log \bar{\E}[\e ^{g_2}],
\end{align*}
where the last step due to the function $\log$ is increasing.
\end{proof}

Let \begin{equation*}
\d \hat{B} _t =\d B_t+ g_t  \d\langle B \rangle _t, \ \ \ 0 \leq t \leq T.
\end{equation*}
Following section 3.2.2, we see that $ {g}\in  M_G^2([0,T]), $ below
we aim to prove
 \begin{equation*}
M_s:=\exp \left\{-\int_0^sg_t{\d}B_t -\frac{1}{2}\int_0^{s} |g_t|^2 \d\langle B \rangle_t \right \},
\end{equation*}
is a uniformly integrable symmetric  $G$-martingale for $s\in[0,T]$.
\begin{lem} \label{lem3.7}
Assume (H1)-(H2).
There holds
\begin{equation}\label{3''}
 \sup_{s\in[0,T)} \bar{\mathbb{E}}[M_s\log M_s]
  \leq\frac{|x-y|^2}{2\alpha\kappa_1^2\lambda_0^\alpha}.
\end{equation}
Consequently, $M_T:=\lim_{s\uparrow T}M_s$ exists and $\{M_s\}_{s\in[0,T]}$ is a uniformly integrable symmetric  $G$-martingale.
\end{lem}
\begin{proof}
Fix $s\in [0,T).$
Applying It\^{o}'s formula to  $|X_t|^2$, we have
\begin{equation*}
|X_t|^2=x^2+\int_0^t\<X_t,b(t,X_t)\>\d t+\int_0^t(\<X_t,h(t,X_t)\>+|\sigma(t,X_t)|^2) \d\langle B \rangle _t
+\int_0^t\<X_t,\sigma(t,X_t)\>\d B_t.
\end{equation*}
Let
\begin{equation*}
\hat{X}_t=x^2+\int_0^t|\<X_t,b(t,X_t)\>|\d t+\int_0^t(|\<X_t,h(t,X_t)\>|+|\sigma(t,X_t)|^2) \d\langle B \rangle _t
 +\int_0^t\<X_t,\sigma(t,X_t)\>\d B_t,
\end{equation*}
and
\begin{align*}
\hat{Y}_t&=y^2+\int_0^t|\<Y_t,b(t,Y_t)\>|\d t+\int_0^t(|\<X_t,h(t,Y_t)-\sigma(t,Y_t)g_t\>| +|\sigma(t,X_t)|^2)\d\langle B \rangle _t\\
&\quad+\int_0^t\<Y_t,\sigma(t,Y_t)\>\d B_t.
\end{align*}
For any $n\geq1,$
let $\hat\tau_n=\inf\{t\in[0,T]| |\hat{X}_t|+|\hat{Y}_t|\geq n\}$. By Lemma \ref{s3}, $\hat\tau_n$ is quasi-continuous, and $X_{t\wedge\hat\tau _n}$, $ Y_{t\wedge\hat\tau_n}$ are bounded,
 which implies $g_{t\wedge\hat\tau_n}$ is bounded.
So for any $n\geq1$ and by the Girsanov theorem in \cite[Theorem 5.2]{HJPS}, $(\hat{B} _t)_{t\in[0,s \wedge\hat\tau _n]}$ is a $G$-Brownian motion under $\hat{\mathbb{E}}_n:=\bar{\mathbb{E}}[\cdot M_{s\wedge\hat\tau _n}]$.

Moreover,   Lemma \ref{lem2.1} implies
$\langle \hat{B}\rangle_t=\langle B \rangle_t. $
Rewrite \eqref{00} and \eqref{1'} as
\begin{eqnarray}\label{Y}
 \d X_t&=&b(t,X_t)\d t+  h(t,X_t)\d\langle \hat{B}\rangle_t+ \sigma(t,X_t)\d  \hat{B}_t-\frac{X_t-Y_t}{\lambda_t^\alpha}\d\langle  \hat{B}\rangle_t, \ X_0=x, \nonumber \\
 \d Y_t&=&b(t,Y_t) \d t+  h(t,Y_t) \d\langle  \hat{B}\rangle_t+ \sigma(t,Y_t)\d  \hat{B}_t, \ Y_0=y.
\end{eqnarray}
Substituting $B_t=\hat{B}_t-\int_0^tg_s\d\<B\>_s$ in the first equation in \eqref{qc1}, using the fact of  $\<\hat{B}\>_t=\<B\>_t$, and
repeating     procedures in  \eqref{qc1}, which yield
\begin{align*}
 \d|X_t-Y_t|^2
 &\leq\left(2K+\frac{2K}{\underline{\sigma}^2}+K^2-\frac{2}{\lambda_t^\alpha}\right)|X_t-Y_t|^2 \d\langle  \hat{B}\rangle_t \\
 &\quad+2\langle X_t-Y_t, \sigma(t,X_t)-\sigma(t,Y_t)\rangle \d  \hat{B}_t.
\end{align*}
So,
\begin{align*}
 \d\frac{|X_t-Y_t|^2}{\lambda_t^\alpha}
 &\leq-\frac{|X_t-Y_t|^2}{(\lambda_t^\alpha)^2}\left(2-2K{\lambda_t^\alpha}-\frac{2K}{\underline{\sigma}^2}\lambda_t^\alpha-K^2{\lambda_t^\alpha}+\frac{1}{\underline{\sigma}^2}(\lambda_t^\alpha)'\right) \d\langle \hat{B}\rangle_t\\
 &\quad+\frac{2}{\lambda_t^\alpha}\langle X_t-Y_t, \sigma(t,X_t)-\sigma(t,Y_t)\rangle \d  \hat{B}_t.
\end{align*}
From \eqref{alpha}, we know that
 $$\alpha=\ff{2\kappa_1^2}{\kappa_2^2} -\triangle\leq 2-\triangle,$$
where $\triangle:=2K{\lambda_t^\alpha}+\frac{2K}{\underline{\sigma}^2}\lambda_t^\alpha+K^2{\lambda_t^\alpha}-\frac{1}{\underline{\sigma}^2}(\lambda_t^\alpha)'$.

Therefore,
\begin{align}\label{4}
 \int_0^{s\wedge\hat\tau _n}\frac{|X_t-Y_t|^2}{(\lambda_t^\alpha)^2} \d\langle  \hat{B}\rangle_t
 &\leq \int_0^s\frac{2}{\alpha\lambda_t^\alpha}\langle X_t-Y_t, \sigma(t,X_t)-\sigma(t,Y_t)\rangle \d  \hat{B}_t \\
 &\quad-\frac{|X_s-Y_s|^2}{\alpha\lambda_s^\alpha}+\frac{|x-y|^2}{\alpha\lambda_0^\alpha}, \ s \in [0,T).\nonumber
\end{align}
Since $(\hat{B}_t)_{t\in{[0,s\wedge\hat\tau _n]}}$ is a $G$-Brownian motion under $\hat{\mathbb{E}}$,
taking expectation $\hat{\mathbb{E}}$ on both sides of \eqref{4},  we obtain
\begin{equation}\label{2}
 \hat{\mathbb{E}}\int_0^{s\wedge\hat\tau _n}\frac{|X_t-Y_t|^2}{(\lambda_t^\alpha)^2} \d\langle  \hat{B}\rangle_t
 \leq\frac{|x-y|^2}{\alpha\lambda_0^\alpha}.
\end{equation}
From the definition of $M_t$, $\hat{B} _t$ and Lemma \ref{lem2.1}, it holds that
\begin{align*}
M_{s\wedge\hat\tau _n}&=\exp \left\{-\int_0^{s\wedge\hat\tau _n}g_t{\d}\hat{B}_t
+\frac{1}{2}\int_0^{s\wedge\hat\tau _n} |g_t|^2{\d}\langle {B} \rangle_{t}\right\}\\
&=\exp \left\{-\int_0^{s\wedge\hat\tau _n}g_t{\d}\hat{B}_t
+\frac{1}{2}\int_0^{s\wedge\hat\tau _n}|g_t|^2{\d}\langle \hat{B} \rangle_{t}\right\}, \  c-q.s..
\end{align*}
By (H2), we have
\begin{equation}\label{logM}
\log M_{s\wedge\hat\tau _n}\leq-\int_0^{s\wedge\hat\tau _n}g_t{\d}\hat{B}_t +\frac{1}{2\kappa_1^2}\int_0^{s\wedge\hat\tau _n} \frac{1}{{(\lambda_t^\alpha)}^2}|(X_t-Y_t)|^2{\d}\langle \hat{B} \rangle_{t}, \  c-q.s..
\end{equation}
It follows \eqref{2} that
\begin{equation}\label{3'}
  \bar{\mathbb{E}}[M_{s\wedge\hat\tau _n}\log M_{s\wedge\hat\tau _n}]= \hat{\mathbb{E}}[\log M_{s\wedge\hat\tau _n}]
  \leq\frac{|x-y|^2}{2\alpha\kappa_1^2\lambda_0^\alpha}, \ s\in[0,T).
\end{equation}
Applying  It\^{o}'s formula to
$M_{s\wedge\hat\tau _n}=\e^{u_{s\wedge\hat\tau _n}}$ for the process
$$u_{s\wedge\hat\tau _n}=-\int_0^{s\wedge\hat\tau _n}g_t{\d} {B}_t-\frac{1}{2}\int_0^{s\wedge\hat\tau _n} |g_t|^2{\d}\langle {B} \rangle_{t},$$ we conclude that
$$\d M_{s\wedge\hat\tau _n}=-\int_{0}^{{s\wedge\hat\tau _n}}g_t\d B_t,$$
thus $\{M_t\}_{t\in[0,s\wedge\hat\tau _n]}$ is a   symmetric $G$-martingale.
From \eqref{3'} and Lemma \ref{um}, $\{M_{s\wedge\hat\tau _n}\}_{s\in [0,T)}$ is a uniformly symmetric $G$-martingale,  thus
$\BE M_s=\lim_{n\rightarrow\infty}\BE M_{s\wedge\hat\tau _n}=1$ by Lemma \ref{lem3.2}.
So that $\{M_t\}_{t\in[0,s)}$ is a symmetric $G$-martingale.

Let $\hat{\E}=\BE[M_s\cdot], \ s\in [0,T)$. Letting $n\rightarrow \infty$, we have $\hat\tau _n\uparrow T$.
  By   the Fatou lemma,
 \begin{align*}
  \lim_{n\rightarrow \infty}\hat{\mathbb{E}}_n[\log M_{s\wedge\hat\tau _n}]
  &= \lim_{n\rightarrow \infty}\hat{\mathbb{E}}[\log M_{s\wedge\hat\tau _n}]
  = \lim_{n\rightarrow \infty}\hat{\E}[\frac{1}{2}\int_0^{s\wedge\hat\tau _n}|g_t|^2{\d}\langle \hat{B} \rangle_{t}]\\
  &\geq\hat{\E}[\frac{1}{2}\int_0^{s}|g_t|^2{\d}\langle \hat{B} \rangle_{t}]
  =\hat{\E}\log M_s.
\end{align*}
Thus
\begin{equation*}
  \bar{\mathbb{E}}[M_s\log M_s]= \hat{\mathbb{E}}[\log M_s]
 \leq\lim_{n\rightarrow \infty}\hat{\mathbb{E}}_n[\log M_{s\wedge\hat\tau _n}]
  \leq\frac{|x-y|^2}{2\alpha\kappa_1^2\lambda_0^\alpha}, \ s\in[0,T].
\end{equation*}
Using Theorem \ref{cts} once again, $\{M_{s}\}_{s\in [0,T]}$ is a  uniformly symmetric  $G$-martingale.

\end{proof}

\begin{lem} \label{lem3.8}
Assume (H1)-(H2).
We  have $X_T=Y_T$, c-q.s..
\end{lem}
\begin{proof}
Let
$$\tau=\inf\{t\in[0,T] \ | \  X_t=Y_t  \}.$$
For any $P\in \mathcal{P}$, define $\hat{\E}_P=\E_{P}[M_T\cdot]$, then $\hat{B}_t$ is a martingale under $\hat{\E}_P.$
If there exists a $\omega\in\Omega$ such that $\tau(\omega)>T$, then
\begin{equation*}
  \inf_{t\in[0,T]}|X_t-Y_t|^2(\omega)>0.
\end{equation*}
So
\begin{equation}\label{contra}
\hat{\mathbb{E}}_P\int_0^T\frac{|X_t-Y_t|^2}{(\lambda_t^\alpha)^2} \d\langle \hat{B}\rangle_t=\infty
\end{equation}
holds on the set $\{\omega|\tau(\omega)>T\}$,  
 which is  a contradiction with \eqref{2}, thus  $\hat{\E}_P$-a.s., $\tau(\omega)\leq T$, then
 $$\hat{\E}_P1_{\{\omega| X_T\neq Y_T\}}=0. $$
  Similar analysis with Lemma \ref{lem2.1}, we have
  $$\hat{\E}1_{\{\omega| X_T \neq Y_T\}}=0.$$
Therefore, $X_T=Y_T$ under $\hat\E$.
\end{proof}

\begin{lem} \label{lem3.9}
Assume (H1)-(H2). Then
\begin{align}\label{aug}
&\sup_{s\in[0,T]} \bar{\mathbb{E}}\left[M_s\exp\left\{ \frac{\alpha^2}{8(\kappa_2-\kappa_1)^2}\int_0^{s}\frac{|X_t-Y_t|^2}{(\lambda_t^\alpha)^2} \d\langle \hat{B}\rangle_t\right\}\right]\\ \nonumber
 &\leq\exp\left\{\frac{\alpha K\left(2+K+\frac{2}{\underline{\si}^2}\right)|x-y|^2}{4(\kappa_2-\kappa_1)^2(\ff{2\kappa_1^2}{\kappa_2^2}-\alpha)\left(1-\e^{-\underline{\sigma}^2K\left(2+K+\frac{2}{\underline{\si}^2}\right)T}\right)}\right\}.
\end{align}
Consequently,
\begin{equation*}
\sup_{s\in[0,T]}\bar{\mathbb{E}}(M_s)^{1+a}
\leq\exp\left\{\frac{\alpha K\left(2+K+\frac{2}{\underline{\si}^2}\right)(\alpha {\kappa_1}+2(\kappa_2-\kappa_1))|x-y|^2}{4(\kappa_2-\kappa_1)^2(\ff{2\kappa_1^2}{\kappa_2^2}-\alpha)(2\alpha {\kappa_1}+2(\kappa_2-\kappa_1))(1-\e^{-\underline{\sigma}^2K\left(2+K+\frac{2}{\underline{\si}^2}\right)T})}\right\}
\end{equation*}
holds for
$$a=\frac{\alpha^2\kappa_1^2}{4(\kappa_2-\kappa_1)^2+4\alpha (\kappa_2-\kappa_1){\kappa_1}}.$$

\end{lem}

\begin{proof}
Let $\tau_m=\inf\{t\in[0,T]\  | \ \int_0^t (\frac{|X_s-Y_s|^2}{(\lambda_s^\alpha)^2}+1  )\d\langle  {B}\rangle_s\geq m\}$.
  Applying Lemma \ref{s3} for  processess $Z_s=0, \eta_s=0,$ and $\zeta_s=\frac{|X_s-Y_s|^2}{(\lambda_s^\alpha)^2}+1$, we know that $\tau_m$ is quasi-continuous.
From \eqref{2'}, we know that $\lim_{m\to\infty}\tau_m=T$.
  By \eqref{4'}, (H2), and Lemma \ref{s4}-\ref{s5}, for some $\delta>0$, we have
\begin{align*}
&\bar{\mathbb{E}}\exp\left\{ \delta\int_0^{s\wedge \tau_m}\frac{|X_t-Y_t|^2}{(\lambda_t^\alpha)^2} \d\langle {B}\rangle_t\right\}\\
 &\leq \bar{\mathbb{E}}\exp\Bigg\{\frac{\delta|x-y|^2}{\alpha\lambda_0^\alpha}
 +\frac{2\delta}{\alpha}\int_0^{s\wedge \tau_m}\frac{1}{\lambda_t^\alpha}\langle X_t-Y_t, \sigma(t,X_t)-\sigma(t,Y_t)\rangle \d {B}_t\Bigg\}\\
 &\leq\exp\left\{\frac{\delta|x-y|^2}{\alpha\lambda_0^\alpha}\right\} \bar{\mathbb{E}}
 \left(\exp\left\{\frac{2(\kappa_2-\kappa_1)\delta}{\alpha}\int_0^{s\wedge \tau_m}\frac{1}{\lambda_t^\alpha} |X_t-Y_t|{\d}{B}_t\right\}\right) \\
  &=\exp\frac{\delta|x-y|^2}{\alpha\lambda_0^\alpha} \bar{\mathbb{E}}
 \Bigg(\frac{8\delta^2(\kappa_2-\kappa_1)^2}{\alpha^2}\int_0^{s\wedge \tau_m}\frac{1}{(\lambda_t^\alpha)^2}| X_t-Y_t|^2 \d \langle{B}\rangle_t
 \Bigg\}\Bigg)^{\frac{1}{2}}.
\end{align*}
Taking $\delta=\frac{\alpha^2}{8(\kappa_2-\kappa_1)^2}$, we arrive at
\begin{align*}
&\bar{\mathbb{E}}\exp\left\{ \frac{\alpha^2}{8(\kappa_2-\kappa_1)^2}\int_0^{s\wedge \tau_m}\frac{|X_t-Y_t|^2}{(\lambda_t^\alpha)^2} \d\langle {B}\rangle_t\right\}\\
 & \leq\exp\left\{\frac{\alpha K\left(2+K+\frac{2}{\underline{\si}^2}\right)|x-y|^2}{4(\kappa_2-\kappa_1)(\ff{2\kappa_1^2}{\kappa_2^2}-\alpha)\left(1-\e^{-\underline{\sigma}^2K\left(2+K+\frac{2}{\underline{\si}^2}\right)T}\right)}\right\}.
\end{align*}
Letting $m\to \infty$, this implies that
\begin{align}\label{eT}
 &\bar{\mathbb{E}}\exp\left\{ \frac{\alpha^2}{8(\kappa_2-\kappa_1)^2}\int_0^{s}\frac{|X_t-Y_t|^2}{(\lambda_t^\alpha)^2} \d\langle {B}\rangle_t\right\}\\
 & \leq\exp\left\{\frac{\alpha K\left(2+K+\frac{2}{\underline{\si}^2}\right)|x-y|^2}{4(\kappa_2-\kappa_1)^2(\ff{2\kappa_1^2}{\kappa_2^2}-\alpha)\left(1-\e^{-\underline{\sigma}^2K\left(2+K+\frac{2}{\underline{\si}^2}\right)T}\right)}\right\},\nonumber
\end{align}
which is \eqref{aug}.

Next,
let $\tilde{\tau}_n=\inf\{t\in [0,T]\ | \int_0^t(\frac{|X_s-Y_s|^2}{(\lambda_s^\alpha)^2}+1 )\d\langle \hat{B}\rangle_s\geq n\}$, similar with $\tau_m$, $\tilde{\tau}_n$ is quasi-continuous. From \eqref{2}, we know that $\lim_{n\to\infty}\tilde{\tau}_n= T$.
 Similar with the process of deducing in    \eqref{eT}, we have
\begin{align}\label{e'}
& \bar{\mathbb{E}}\left[M_s\exp\left\{ \frac{\alpha^2}{8(\kappa_2-\kappa_1)^2}\int_0^{s}\frac{|X_t-Y_t|^2}{(\lambda_t^\alpha)^2} \d\langle \hat{B}\rangle_t\right\}\right]\\ \nonumber
&= \hat{\mathbb{E}}\exp\left\{ \frac{\alpha^2}{8(\kappa_2-\kappa_1)^2}\int_0^{s}\frac{|X_t-Y_t|^2}{(\lambda_t^\alpha)^2} \d\langle \hat{B}\rangle_t\right\}\\ \nonumber
  &\leq\exp\left\{\frac{\alpha K\left(2+K+\frac{2}{\underline{\si}^2}\right)|x-y|^2}{4(\kappa_2-\kappa_1)^2(\ff{2\kappa_1^2}{\kappa_2^2}-\alpha)\left(1-\e^{-\underline{\sigma}^2K\left(2+K+\frac{2}{\underline{\si}^2}\right)T}\right)}\right\}.
\end{align}
Moreover,
\begin{align*}
&\bar{\mathbb{E}}(M_s)^{1+a}=\hat{\mathbb{E}}(M_s)^a\\
&=\hat{\mathbb{E}}\exp \Bigg\{-a\int_0^sg_t{\d}\hat{B}_t +\frac{a}{2}\int_0^s |g_t|^2{d}\langle \hat{B}\rangle_t\Bigg\}\\
&=\hat{\mathbb{E}}\exp \Bigg\{-a\int_0^sg_t{\d}\hat{B}_t -\frac{a^2q}{2}\int_0^s |g_t|^2{d}\langle \hat{B}\rangle_t+\frac{a(aq+1)}{2}\int_0^s |g_t|^2{d}\langle \hat{B}\rangle_t\Bigg\}\\
&\leq\Bigg(\hat{\mathbb{E}}\exp \Bigg\{-aq\int_0^sg_t{\d}\hat{B}_t-\frac{a^2q^2}{2}\int_0^s |g_t|^2{d}\langle \hat{B}\rangle_t\Bigg\}\Bigg)^{\frac{1}{q}}\\
&\quad\times\Bigg(\hat{\mathbb{E}}\exp \Bigg\{\frac{aq(aq+1)}{2(q-1)}\int_0^s |g_t|^2{d}\langle \hat{B}\rangle_t \Bigg\}\Bigg)^{\frac{q-1}{q}}\\
&=\Bigg(\hat{\mathbb{E}}\exp \Bigg\{\frac{aq(aq+1)}{2(q-1)}\int_0^s |g_t|^2{d}\langle \hat{B}\rangle_t \Bigg\}\Bigg)^{\frac{q-1}{q}}.
\end{align*}
From (H2), we have
\begin{equation}\label{ast}
\bar{\mathbb{E}}(M_s)^{1+a}
\leq\Bigg(\hat{\mathbb{E}}\exp \Bigg\{\frac{aq(aq+1)}{2\kappa_1^2(q-1)}\int_0^s \frac{1}{{(\lambda_t^\alpha)}^2}| X_t-Y_t|^2{d}\langle \hat{B}\rangle_t \Bigg\}\Bigg)^{\frac{q-1}{q}}.
\end{equation}
Taking $q=1+\sqrt{1+a^{-1}}$, it holds that
\begin{align}\label{ast'}
\frac{aq(aq+1)}{2\kappa_1^2(q-1)}&=\frac{(a+\sqrt{a(a+1)})(a+1+\sqrt{a(a+1)})}{2\kappa_1^2\sqrt{1+a^{-1}}}\\ \nonumber
&=\frac{(a+\sqrt{a(a+1)})^2}{2\kappa_1^2}\\ \nonumber
&=\frac{\alpha^2}{8(\kappa_2-\kappa_1)^2}.\nonumber
\end{align}
Then,
\begin{equation*}
  \frac{q-1}{q}=\frac{\sqrt{1+a^{-1}}}{1+\sqrt{1+a^{-1}}}=\frac{\alpha {\kappa_1}+2(\kappa_2-\kappa_1)}{2\alpha {\kappa_1}+2(\kappa_2-\kappa_1)}.
\end{equation*}
Therefore, by recalling the expressions  \eqref{e'} -- \eqref{ast'}, we get
\begin{equation}\label{M1}
\bar{\mathbb{E}}(M_s)^{1+a}
\leq\exp\left\{\frac{\alpha K\left(2+K+\frac{2}{\underline{\si}^2}\right)(\alpha {\kappa_1}+2(\kappa_2-\kappa_1))|x-y|^2}{4(\kappa_2-\kappa_1)^2(\ff{2\kappa_1^2}{\kappa_2^2}-\alpha)(2\alpha {\kappa_1}+2(\kappa_2-\kappa_1))(1-\e^{-\underline{\sigma}^2K\left(2+K+\frac{2}{\underline{\si}^2}\right)T})}\right\},
\end{equation}
this completes the proof.

\end{proof}

\subsubsection{Proof to Theorem \ref{th1}}
\begin{itemize}
  \item [(1)] Lemma \ref{lem3.7} ensures that under $\hat \E:=\bar{\mathbb{E}}[\cdot M_{T}]$, $\{\hat{B}_t\}_{t\in[0,T]}$ is  a $G$-Brownian motion, and
  \begin{equation*}
  \bar{\mathbb{E}}[M_T\log M_T]
  \leq\frac{|x-y|^2}{2\alpha\kappa_1^2\lambda_0^\alpha}.
\end{equation*}
Then by \eqref{00} and \eqref{Y}, the coupling $(X_t, Y_t)$ is well constructed under $\hat\E$ for $t\in[0,T]$. Moreover, due to Lemma \ref{lem3.8}, $X_T=Y_T$ holds $\hat\E$-q.s., which fits well the requirement of coupling by
change of measure.
Since for all $P\in \mathcal{P}$, ${\E}_P[M_T]=1$,  by Young's inequality in Lemma \ref{Young}, 
 for any $f\in C_b^+(\R)$,
we obtain
\begin{align*}
\bar{P}_T \log f(y)&=\bar{\mathbb{E}}[\log f(X_T^y)]=\hat{\mathbb{E}}[\log f(Y_T^y)]=\hat{\mathbb{E}}[\log f(X_T^x)]=\bar{\mathbb{E}}[M_T\log f(X_T^x)]\\
&\leq \log\bar{\mathbb{E}}[f(X_T^x)]+\bar{\mathbb{E}}[M_T\log M_T]\\
&=\log\bar{P}_Tf(x)+ \bar{\mathbb{E}}[M_T\log M_T]\\
&\leq\log\bar{P}_Tf(x)+\frac{|x-y|^2}{2\alpha\kappa_1^2\lambda_0^\alpha}\\
&=\log\bar{P}_Tf(x)+\frac{K\left(2+K+\frac{2}{\underline{\si}^2}\right)|x-y|^2}{2\alpha\kappa_1^2(\ff{2\kappa_1^2}{\kappa_2^2}-\alpha)(1-\e^{-\underline{\sigma}^2K\left(2+K+\frac{2}{\underline{\si}^2}\right)T})}.
\end{align*}
For $\alpha\in(0,\ff{2 \kappa_1^2}{\kappa_2^2}),$  
   taking $\alpha=\ff{ \kappa_1^2}{\kappa_2^2} $,   (1) of Theorem \ref{th1} holds.
 \item [(2)]

Taking $\alpha=\frac{2(\kappa_2-\kappa_1)}{ {\kappa_1}(\sqrt{p}-1)} $ in \eqref{M1} which is in $(0,\ff{2\kappa_1^2}{\kappa_2^2})$
 for $p>(1+\ff{\kappa_2^3-\kappa_1\kappa_2^2}{\kappa_1^3})^2$,
 we have $\frac{p}{p-1}=1+a$, by Lemma \ref{lem3.9}, this leads to
\begin{align*}
&(\bar{\mathbb{E}}M_T^{\frac{p}{p-1}})^{p-1}=(\bar{\mathbb{E}}M_T^{1+a})^{p-1}=(\hat{\mathbb{E}}M_T^a)^{p-1}\\
&\leq\exp\left\{\frac{(p-1) K\left(2+K+\frac{2}{\underline{\si}^2}\right)\alpha(\alpha {\kappa_1}+2(\kappa_2-\kappa_1))|x-y|^2}{4(\kappa_2-\kappa_1)^2(\ff{2\kappa_1^2}{\kappa_2^2}-\alpha)(2\alpha {\kappa_1}+2(\kappa_2-\kappa_1))\left(1-\e^{-\underline{\sigma}^2K\left(2+K+\frac{2}{\underline{\si}^2}\right)T}\right)}\right\}\\
&=\exp\left\{\frac{\sqrt{p}(\sqrt{p}-1) K\left(2+K+\frac{2}{\underline{\si}^2}\right)|x-y|^2}{4(\kappa_2-\kappa_1)[ {\kappa_1}(\sqrt{p}-1)-(\kappa_2-\kappa_1)]\left(1-\e^{-\underline{\sigma}^2K\left(2+K+\frac{2}{\underline{\si}^2}\right)T}\right)}\right\}.
\end{align*}
Thus, due to  H\"{o}lder's inequality, for any $f\in C_b^+(\R)$,
\begin{align*}
&(\bar{P}_Tf)^p(y)=(\bar{\mathbb{E}}f(X_T^y))^p=(\hat{\mathbb{E}}f(Y_T^y))^p =(\hat{\mathbb{E}}f(X_T^x))^p  =(\bar{\mathbb{E}}M_Tf(X_T^x))^p\nonumber\\
&\leq (\bar{\mathbb{E}}f^p(X_T^x))\left(\bar{\mathbb{E}}\left[M_T^{\frac{p}{p-1}}\right]\right)^{p-1}\\
&\leq \bar{P}_Tf^p(x)\exp\left\{\frac{\sqrt{p}(\sqrt{p}-1) K\left(2+K+\frac{2}{\underline{\si}^2}\right)|x-y|^2}{4(\kappa_2-\kappa_1)[ {\kappa_1}(\sqrt{p}-1)-(\kappa_2-\kappa_1)](1-\e^{-\underline{\sigma}^2K\left(2+K+\frac{2}{\underline{\si}^2}\right)T})}\right\},
\end{align*}
which is the result (2) of Theorem \ref{th1}.
\end{itemize}

\subsection{Gradient Estimate}
Due to the lack of additivity of  $G$-expectation, neither from  the Bismut formula \cite[(1.8), (1.14)]{W1} by  coupling by change of measure  to   get gradient estimate, nor    Malliavin calculus in the $G$-SDEs. Instead, we directly to estimate the local Lipschitz constant defined below.
For a real-valued function $f$ defined on a metric sapce $(H, \rho)$, define
\begin{align}\label{mod}
|\nabla f(z)|=\limsup_{x\to z}\frac{|f(x)-f(z)|}{\rho(x,z)}, \ \ z\in H.
\end{align}
Then $|\nabla f(z)|$ is called the local Lipschitz constant of $f$ at point $z\in H$.

\begin{thm} Assume (H1)-(H2). Then
 for every  $f \in {C}_b^+(\mathbb{R})$, it holds that
\begin{equation}\label{thm4.1}
 \|\nabla\bar{P}_T f\|_\infty\leq \|f\|_\infty\frac{2}{\kappa_1\sqrt{\alpha\lambda_0^\alpha}},
  \end{equation}
where $\lambda_0^\alpha$ is defined in \eqref{lambda} for $t=0$.
\end{thm}
\begin{proof}
By the proof of Theorem \ref{th1}, we have
 \begin{align*}
  |\bar{P}_Tf(y)-\bar{P}_Tf(x) |
  &=|\bar{\mathbb{E}}f(X_T^y)-\bar{\mathbb{E}}f(X_T^x) |\\
   &=|\bar{\mathbb{E}}M_Tf(X_T^x)-\bar{\mathbb{E}}f(X_T^x)|\\
   &\leq\|f\|_\infty(\bar{\mathbb{E}}[|M_T-1|]).\\
   \end{align*}
Noting that $|x-1|\leq (x+1)|\log x|$ for any $x>0$, then
 \begin{align}\label{pp0}
  |\bar{P}_Tf(y)-\bar{P}_Tf(x) |
   &\leq\|f\|_\infty(\bar{\mathbb{E}}[(M_T+1)\log M_T])\\ \nonumber
   &=\|f\|_\infty\left(\hat{\mathbb{E}}[|\log M_T|]+\bar{\mathbb{E}}[ |\log M_T|]\right)
   \end{align}
  From \eqref{logM} and \eqref{3'},  it holds that
    \begin{align*}
   \hat{\mathbb{E}}[|\log M_T|]
   &\leq
   \hat{\mathbb{E}} \left[ \left|\int_0^Tg_t{\d}\hat{B}_t \right|\right]
   +  \hat{\mathbb{E}} \left[ \frac{1}{2\kappa_1^2}\int_0^T \frac{1}{{(\lambda_t^\alpha)}^2}|(X_t-Y_t)|^2{d}\langle \hat{B} \rangle_{t} \right]\\
    &\leq
   \hat{\mathbb{E}} \left[\int_0^T\frac{1}{(\lambda_t^\alpha)^2}\left|\ff{1}{\kappa_1} (X_t-Y_t)\right|^2{d}\langle\hat{B}\rangle_t\right]^{\ff{1}{2}}
   +  \hat{\mathbb{E}} \left[ \frac{1}{2\kappa_1^2}\int_0^T \frac{1}{{(\lambda_t^\alpha)}^2}|(X_t-Y_t)|^2{d}\langle \hat{B} \rangle_{t} \right]\\
   &\leq\frac{1}{\kappa_1\sqrt{\alpha\lambda_0^\alpha}}|x-y|+\frac{1}{ 2\alpha\kappa_1^2\lambda_0^\alpha}|x-y|^2.
    \end{align*}
Similarly, we obtain
  \begin{align*}
   \bar{\mathbb{E}}[|\log M_T|]
   &\leq
    \bar{\mathbb{E}} \left[ \left|\int_0^Tg_t{d} {B}_t \right|\right]
   +   \bar{\mathbb{E}} \left[ \frac{1}{2\kappa_1^2}\int_0^T \frac{1}{{(\lambda_t^\alpha)}^2}|(X_t-Y_t)|^2{d}\langle  {B} \rangle_{t } \right]\\
    &\leq
    \bar{\mathbb{E}} \left[\int_0^T\frac{1}{(\lambda_t^\alpha)^2}\left|\ff{1}{\kappa_1} (X_t-Y_t)\right|^2{d}\langle {B}\rangle_t\right]^{\ff{1}{2}}
   +   \bar{\mathbb{E}} \left[ \frac{1}{2\kappa_1^2}\int_0^T \frac{1}{{(\lambda_t^\alpha)}^2}|(X_t-Y_t)|^2{d}\langle  {B} \rangle_{t } \right]\\
   &\leq\frac{1}{\kappa_1\sqrt{\alpha\lambda_0^\alpha}}|x-y|+\frac{1}{ 2\alpha\kappa_1^2\lambda_0^\alpha}|x-y|^2.
    \end{align*}
  It follows from \eqref{pp0} that
   \begin{eqnarray}\label{pp1}
  |\bar{P}_Tf(y)-\bar{P}_Tf(x) | \leq \|f\|_\infty \left(\frac{2}{\kappa_1\sqrt{\alpha\lambda_0^\alpha}}|x-y|+\frac{1}{ \alpha\kappa_1^2\lambda_0^\alpha}|x-y|^2\right).
   \end{eqnarray}
   This together with \eqref{mod} yields
   \begin{align}\label{ge1}
   |\nabla\bar{P}_T f(x)|\leq \|f\|_\infty\frac{2}{\kappa_1\sqrt{\alpha\lambda_0^\alpha}},
\end{align}
   which implies \eqref{thm4.1}.

  \end{proof}

\section{Appendix--The quasi-continuity of stopping times}
This part is essentially from \cite{song, song2}. To make the content self-contained, we cite  some results from  \cite{song, song2} and restated them as follows.
\begin{lem}\label{s1}\rm(\cite[Lemma 3.3]{song2})
Let $E$ be a metric space and a mapping $E \times [0, T] \ni (\omega, t) \to M_t(\omega) \in \R$ be continuous on $E \times [0, T]$.
Define $\tau_a=\inf\{t>0|M_t>a\}\wedge T$ and $\underline{\tau}_a=\inf\{t>0|M_t\geq a\}\wedge T$.  Then
$-\tau_a$ and $\underline{\tau}_a$ are both lower semi-continuous.
\end{lem}

\begin{lem}\label{s2}\rm(\cite[Lemma 3.4]{song2})
For any closed set $F\subset\Omega_T$, we have
$$c(F)=\inf\{ c(O)| F\subset O,\ O \  is \ open\},$$
where $c$ is the capacity induced by $\bar\E$.
\end{lem}
The following lemma plays a crucial role in studying the  quasi-continuity of stopping times under nonlinear expectation space, which is a dramatic different with classic linear expectation space. For reader's convenience, we give the proof of the lemma.
\begin{lem}\label{s3}\rm(\cite[Lemma 4.3]{song})
Let $Y_t = \int_0^t\<Z_s,\d B_s\>+\int_0^t \eta_s\d s+\int_0^t  tr[\zeta_s\d \<B\>_s]$ with $Z\in[H_G^1([0,T])]^d$ and $\eta, \zeta^{i,j}\in M_G^1([0,T])$. Assume $\int_0^t \eta_s\d s+\int_0^t  tr[\zeta_s\d \<B\>_s]$ is non-decreasing and $$\int_0^t  tr[Z_sZ_s^\ast\d \<B\>_s]+\int_0^t \eta_s\d s+\int_0^t  tr[\zeta_s\d \<B\>_s]$$
is strictly increasing. Then, for $a > 0$, $\tau_a := \inf\{t \geq 0| Y_t > a\} \wedge T$ is    quasi-continuous.
\end{lem}
\begin{proof}
Let $\underline{\tau}_a = \inf\{t \geq 0| Y_t \geq a\} \wedge T$.
 Since $Y$ is quasi-continuous, then for all  $\epsilon > 0$, there
exists an open set $O_1$ with $c(O_1) < \ff{\epsilon}{2} $ such that $Y_{\cdot}(\cdot)$ is continuous on $O_1^c\times [0, T].$
Define
$$S_a(Y) = \{\omega \in  \Omega_T|  \ there \ exists \ (r, s) \in  Q_T \ s.t. \ Y_t(\omega) = a \  for \  all \ t \in  [s, r]\},$$
where
$$Q_T = \{(r, s)| T \geq r > s \geq  0, \  r, s \in \mathbb{Q}\},  \ and \  \mathbb{Q}  \  is \ the  \ totality \ of \ rational  \ numbers.$$
We   divide the proof   into following  five steps.
\begin{itemize}
  \item [(1)]We first prove $[\tau_a > \underline{\tau}_a] \subset  S_a(Y) \bigcup \cup_{r\in \mathbb{Q}\bigcap[0,T ]}[Y_{r\wedge\tau_a} < Y_{r\wedge\underline{\tau}_a}]=:A.$

It is  equivalent to prove $[\tau_a > \underline{\tau}_a] \subset  S_a(Y) + A\setminus S_a(Y).$

For any $\omega\in [\tau_a > \underline{\tau}_a], $ i.e.,  for any $\omega$ with $\tau_a(\omega) > \underline{\tau}_a(\omega)$, if $\omega \in S_a(Y)$, which ends the proof. If $\omega \notin S_a(Y)$, i.e.,
for any   $(r, s) \in  Q_T$, there exists  a $t \in  [s, r]$, s.t.   $Y_t(\omega)\neq a$. Since $\mathbb{Q}$ is dense in $\R$, and $\tau_a\geq \underline{\tau}_a$, it's clear that
$\omega\in A\setminus S_a(Y).$
  \item [(2)]We claim that $c(S_a(Y))=0$.
  \begin{enumerate}
    \item [(i)] If $Z=0$, then $Y_t$ is strictly increasing, thus $\tau_a= \underline{\tau}_a$, which implies $c(S_a(Y))=0$.
    \item[(ii)] If $Z\neq0$, since $B_t$ with infinite variation, it is impossible for $Y_t=a, t\in [s, r],$  then $c(S_a(Y))=0$.
  \end{enumerate}
  \item [(3)]We claim that $c(A)=0$.

   Noting that $Y_{r\wedge\tau_a} \leq Y_{r\wedge\underline{\tau}_a}$    and
   \begin{align}\label{tau1}
     \BE[Y_{r\wedge\tau_a} - Y_{r\wedge\underline{\tau}_a}]
     & = \BE\left[\int_{r\wedge\underline{\tau}_a}^{r\wedge\tau_a}\<{Z_s},\d B_s\>
     +\int_{r\wedge\underline{\tau}_a}^{r\wedge\tau_a} {\eta_s}\d s+\int_{r\wedge\underline{\tau}_a}^{r\wedge\tau_a}  tr[{\zeta_s}\d \<B\>_s] \right] \\ \no
     & = \BE\left[\int_{r\wedge\underline{\tau}_a}^{r\wedge\tau_a} {\eta_s}\d s+\int_{r\wedge\underline{\tau}_a}^{r\wedge\tau_a}  tr[{\zeta_s}\d \<B\>_s] \right]\no.
     \end{align}
     For $r\leq \underline{\tau}_a$ and $r\geq\tau_a$, it hold that $\BE[Y_{r\wedge\tau_a} - Y_{r\wedge\underline{\tau}_a}]=0.$
     For $ \underline{\tau}_a< r< {\tau}_a',$ by \eqref{tau1}, we have
   \begin{equation*}
     \BE[Y_{r\wedge\tau_a} - Y_{r\wedge\underline{\tau}_a}]  = \BE\left[\int_{\underline{\tau}_a}^{r} {\eta_s}\d s+\int_{\underline{\tau}_a}^{r}  tr[{\zeta_s}\d \<B\>_s] \right].
     \end{equation*}
     From the assumption of  non-decreasing for $\int_0^t \eta_s\d s+\int_0^t  tr[\zeta_s\d \<B\>_s]$, we derive that  $\BE[Y_{r\wedge\tau_a} - Y_{r\wedge\underline{\tau}_a}]\geq0.$
     By the fact that $Y_{r\wedge\tau_a} \leq Y_{r\wedge\underline{\tau}_a}$ and $\BE[Y_{r\wedge\tau_a} - Y_{r\wedge\underline{\tau}_a}]\geq0,$  we know that $Y_{r\wedge\tau_a} = Y_{r\wedge\underline{\tau}_a}, $ q.s..
     Since $\mathbb{Q}$ is countable,
     then $c(A)=0$

  \item [(4)]$A\cap O_1^c$ is an open set under the topology induced by $O_1^c$.

  Since $Y_{\cdot}(\cdot)$ is continuous on $O_1^c\times [0, T]$, by Lemma \ref{s1},  $\underline{\tau}_a$ (${\tau}_a$) is lower (upper) semi-continuous  on $O_1^c$, then  $Y_{r\wedge\underline{\tau}_a}$ ($Y_{r\wedge{\tau}_a}$)
  is lower  (upper) semi-continuous on $O_1^c$, which means that
  $[Y_{r\wedge\tau_a} < Y_{r\wedge\underline{\tau}_a}]\cap O_1^c$ is an open set under the topology induced by $O_1^c$. 
 Since the union of any collection of open sets in $O_1^c$ is open, then we prove it.

  \item [(5)] $S_a(Y)$ can be covered by countable open sets with capacity small enough.

  By the definition of $S_a(Y)$, we have
  $$S_a(Y)=\bigcup_{(r, s)\in  Q_T}\bigcap_{t\in[s,r]}\{ \omega | Y_t(\omega)=a\}.$$
   Since $Y_{\cdot}(\cdot)$ is continuous on $O_1^c\times [0, T]$,   $\{ \omega | Y_t(\omega)=a\}\cap O_1^c$ is a closed  set under the topology induced by $O_1^c$ for any $t\in[0,T]$.
    Moreover,  $\{ \omega | Y_t(\omega)=a\}$ is a closed  set  as $O_1^c$ is closed.
Then $\bigcap_{t\in[s,r]}\{ \omega | Y_t(\omega)=a\}$ is closed. By Lemma \ref{s2} and  the fact that $c(S_a(Y))=0$, for all  $\epsilon > 0$,   there
exists an open set $O_2^{s,r}$ with $0\leq c(O_2^{s,r}) < \ff{\epsilon }{2^{n+1}}$ such that $\bigcap_{t\in[s,r]}\{ \omega | Y_t(\omega)=a\}\subset O_2^{s,r}$. Let $O_2=\bigcup_{(r, s)\in  Q_T}O_2^{s,r},$ then
$$S_a(Y)\subset O_2, \ c(O_2)<\ff{\epsilon}{2},$$
where $O_2$ is open.

\end{itemize}
Combining (1)--(5), we know that
$$[\tau_a > \underline{\tau}_a] \subset   O_2 \cup A,$$
where $ O_2$ is open under topology induced by $\Omega_T$  and $ A\cap O_1^c$ is   open  under the topology induced by $O_1^c$.
So, there exists an open set $O_3\subset \Omega_T$, such that
$$ A\cap O_1^c=O_3\cap O_1^c\subset O_3.$$
Noting that
\begin{align*}
 A
  &=(A\cap  O_1)\cup  ( A\cap O_1^c)\\
  &\subset O_1\cup( O_3\cap O_1^c)\\
   &\subset O_1\cup O_3.
\end{align*}
Moreover,
$O_3=(O_3\cap O_1)\cup (O_3\cap O_1^c),$
by $c(O_3\cap O_1^c)=0$ of (3), we have
$$c(O_3)\leq c(O_3\cap O_1)+ (O_3\cap O_1^c)<\epsilon.$$
Therefore,
$$[\tau_a > \underline{\tau}_a] \subset  O_2\cup O_1\cup O_3,$$
where $c( O_2\cup O_1\cup O_3)\leq c( O_1)+c(O_2)+c(O_3)\leq 2\epsilon.$
It is clear that
$$[\tau_a > \underline{\tau}_a]^c= [\tau_a\leq\underline{\tau}_a]=[\tau_a=\underline{\tau}_a]\supset
 \left(O_2 \cup O_1\cup O_3\right)^c,$$
thus
 $$[\tau_a=\underline{\tau}_a]\cap O_1^c \supset
 \left(O_2 \cup  O_1\cup O_3\right)^c \cap O_1^c= \left(O_1\cup O_2 \cup O_3\right)^c,$$
By Lemma \ref{s1},     $\tau_a$ is continuous on $[\tau_a=\underline{\tau}_a]\cap O_1^c$.
Therefore, for all $\epsilon>0$, for the open set, $O_1\cup O_2\cup O_3$,
with $c(O_1\cup O_2\cup O_3)<2\epsilon,$
$\tau_a$ is continuous on $(O_1\cup O_2\cup O_3)^c,$
which implies that $\tau_a$ is quasi continuous by  Definition \ref{qc}.
\end{proof}
\begin{lem}\label{s4}\rm(\cite[Proposition 4.10]{liu})
Let $\tau\leq T $ be a quasi-continuous stopping time. Then for each $p \geq 1$, we have
$I_{[0,\tau ]} \in  M_G^ p ([0, T]).$
\end{lem}
\begin{lem}\label{s5}\rm(\cite[Remark 4.12]{liu})
Let $\tau\leq T $ be a quasi-continuous stopping time and $\eta\in M_G^ p ([0, T])$. Then for each $p \geq 1$, we have
$\eta I_{[0,\tau ]} \in  M_G^ p ([0, T]).$
\end{lem}
According to \cite{Li}, for a stopping time $\tau\leq T $,   and $\eta\in M_G^ p ([0, T])$, it holds that
\begin{equation*}
\int_{0}^{ \tau}\eta_{s}dB_{s}
=\int_{0}^{T}\eta_{s}I_{[0,\tau]}(s)dB_{s}.
\end{equation*}

\paragraph{Acknowledgement.}
The authors are grateful to Professor Feng-Yu Wang for his  guidance and helpful comments, as well as Yongsheng  Song and   Xing Huang  for their patient helps, valuable suggestions and corrections.

\beg{thebibliography}{99}

\bibitem{bao} J. Bao, F.-Y. Wang, C. Yuan, Derivative formula and Harnack inequality for degenerate functionals SDEs,   Stoch. Dyn. 13 (2013), 1--22.

\bibitem{cohen} S. Cohen, S. Ji, and S. Peng, Sublinear expectations and Martingales in discrete time,   arXiv:1104.5390.

\bibitem{Denis} L. Denis, M. Hu, S. Peng, Function spaces and capacity related to a sublinear expectation: application to $G$-Brownian motion pathes, Potential Anal. 34 (2011) 139--161.

\bibitem{HJPS} M. Hu, S. Ji, S. Peng, Y. Song,  Comparison theorem, Feynman-Kac formula and Girsanov transformation for BSDEs driven by $G$-Brownian motion,  Stochastic Process. Appl. 124 (2014), 1170-1195.

\bibitem{HWZ} M. Hu, F. Wang,  G. Zheng,  Quasi-continuous random variables and processes under the $G$-expectation framework, Stochastic Process. Appl. 126 (2016),  2367--2387.

\bibitem{XY}X. Huang, F.-F. Yang, Harnack inequality and gradient estimate for (functional)
$G$-SDEs with degenerate noise, (2019) arXiv:1812.04290.

\bibitem{Li} X. Li, S. Peng, Stopping times and related It\^{o}'s calculus with $G$-Brownian motion, Stochastic Process, Appl. 121 (2011)  1492--1508.

\bibitem{liu} G. Liu, Exit times for semimartingales under nonlinear expectation, (2019) arXiv:1812.00838.

\bibitem{Osuka} E. Osuka, Girsanov's formula for $G$-Brownian motion, Stochastic Process. Appl. 123 (2013) 1301--1318.

\bibitem{peng2} S. Peng, $G$-Brownian motion and dynamic risk measures under volatility
uncertainty, (2007) arXiv: 0711.2834v1.

\bibitem{peng1} S. Peng, $G$-expectation, $G$-Brownian motion and related stochastic
calculus of It\^{o} type, in: Stochastic Analysis and Applications, in: Abel Symp., vol. 2, Springer, Berlin, 2007, pp.541--567.

\bibitem{peng4} S. Peng, Nonlinear expectations and stochastic calculus under uncertainty, (2010), arXiv: 1002.4546v1.

\bibitem{RY}  P. Ren, F.-F. Yang, Path independence of additive functionals for stochastic differential equations under
$G$-framework, Front. Math. China 14 (2019) 135--148.

\bibitem{song} Y. Song,   Gradient estimates for nonlinear diffusion semigroups by coupling methods, Sci. China Math., https://doi.org/10.1007/s11425-018-9541-6. 
\bibitem{song2}   Y. Song, Properties of hitting times for $G$-martingales and their applications, Stochastic Process. Appl. 121 (2011)  1770--1784.

\bibitem{song1}  Y. Song, Some properties on $G$-evaluation and it's applications to $G$-martingale decomposition, Sci. China Math.  54 (2011)  287--300.

\bibitem{W3} F.-Y. Wang,  Estimates for invariant probability measures of degenerate SPDEs with singular and path-dependent drifts, Probab. Theory Related Fields 172 (2018), no. 3--4, 1181--1214.
\bibitem{W1} F.-Y. Wang, Harnack inequalities for stochastic partial differential equations,
Springer Briefs in Mathematics, Springer, New York, 2013, pp, ISBN: 978--1--4614--7933-8, 978--1--4614--7934--5.
\bibitem{W4} F.-Y. Wang,
Harnack inequality for SDE with multiplicative noise and extension to Neumann semigroup on nonconvex manifolds, (English summary)
Ann. Probab.  39 (2011),  1449--1467.

\bibitem{W} F.-Y. Wang, Logarithmic Sobolev inequalities on noncompact Riemannian manifolds, Probab. Theory Related Fields 109 (1997) 417--424.

\bibitem{Yang} F.-F. Yang,  Harnack inequality and applications for SDEs driven by $G$-Brownian motion, (2018) arXiv:1808.08712.

\bibitem{Xu} J. Xu,  H. Shang,  B. Zhang, A Girsanov type theorem under $G$-framework,
Stoch. Anal. Appl. 29 (2011) 386--406.

\end{thebibliography}

\end{document}